\documentclass[12pt,reqno]{amsart}

\usepackage{amsmath,amssymb}
\usepackage{amsthm}
\usepackage{thmtools}
\usepackage{mathtools}
\usepackage{dsfont}
\usepackage{color}
\usepackage{enumitem}
\usepackage{hyperref}
\usepackage[inner=1.0in,outer=1.0in,bottom=1.0in, top=1.0in]{geometry}

\usepackage{comment}
\usepackage{cases}

\newcommand{\Z}{{\mathbb{Z}}}
\newcommand{\Q}{{\mathbb{Q}}}
\newcommand{\R}{{\mathbb{R}}}
\newcommand{\C}{{\mathbb{C}}}
\newcommand{\HH}{{\mathbb{H}}}

\def\PSL{{\rm PSL}}

\DeclareMathOperator{\Tr}{Tr}

\DeclareMathOperator{\re}{Re}
\DeclareMathOperator{\im}{Im}

\newcommand\be{\begin{equation}}
	\newcommand\ee{\end{equation}}
\newcommand\bee{\begin{equation*}}
	\newcommand\eee{\end{equation*}}
\newcommand\ben{\begin{enumerate}}
	\newcommand\een{\end{enumerate}}
\newcommand{\defeq}{\vcentcolon=}
\def\SL{{\rm SL}}
\def\GL{{\rm GL}}
\newcommand{\floor}[1]{\left\lfloor #1 \right\rfloor}

\newcommand{\ceil}[1]{\left\lceil #1 \right\rceil}

\renewcommand{\(}{\left(}
\renewcommand{\)}{\right)}
\newcommand{\la}{\left|}
\newcommand{\ra}{\right|}

\renewcommand{\sl}{\big|}

\newcommand{\Mod}[1]{\ (\mathrm{mod}\ #1)}
\newcommand{\pr}{\mathrm{pr}}
\newcommand{\ep}{\varepsilon}

\newtheorem{theorem}{Theorem}

\newtheorem{lemma}[theorem]{Lemma}
\newtheorem{corollary}[theorem]{Corollary}
\newtheorem{proposition}[theorem]{Proposition}

\theoremstyle{remark}
\newtheorem*{remark}{Remark}

\numberwithin{equation}{section}
\numberwithin{theorem}{section}
\numberwithin{lemma}{section}
\numberwithin{proposition}{section} 
\numberwithin{example}{section}
\numberwithin{definition}{section}
\numberwithin{corollary}{section}

\mathtoolsset{showonlyrefs}

\makeatletter
\@namedef{subjclassname@2020}{\textup{2020} Mathematics Subject Classification}
\makeatother

\usepackage[foot]{amsaddr}

\author{Oscar E. Gonz\'alez}

\author{Qihang Sun}

\address{Department of Mathematics, University of Illinois, Urbana, Illinois, United States. }
\email{oscareg2@illinois.edu, qihangs2@illinois.edu}

\title{Effective estimates for traces of singular moduli}

\date{April 6, 2023. }

\begin{document}

	\begin{abstract}
		Traces of singular moduli can be approximated by exponential sums of quadratic irrationals. Recently Andersen and Duke used theory of Maass forms to estimate generalized twisted traces with power-saving error bounds. We establish an asymptotic formula with effective error bounds for such traces. Our methods depend on an explicit bound for sums of Kloosterman sums on $\Gamma_0(4)$. 
	\end{abstract}
 
 \keywords{Traces of singular moduli; Kohnen's plus space; Kloosterman sums;}
 
 \subjclass[2020]{11F37, 11F72, 11L05}

 \maketitle

	\section{Introduction}

	Define the modular $j$-function by
	\begin{equation}
		j(z) \defeq q^{-1} \prod\limits_{n=1}^\infty (1-q^n)^{-24}\Big(1+240\sum\limits_{n=1}^\infty \sum\limits_{m\mid n}m^3 q^n\Big)^3
		= \sum_{n=-1}^\infty c(n)q^n, \qquad q=e^{2\pi i z}.
	\end{equation}
	A problem related to the $j$-function is estimating the traces of singular moduli, which are the traces of $j$ at quadratic irrationals. 
	For a given discriminant $d<0$, let $K$ be the corresponding quadratic field over $\Q$ and let $h(d)$ be its class number. A 
	positive definite integral
	quadratic form is denoted as $Q(x,y)\defeq ax^2+bxy+cy^2=[a,b,c]$ for some $a,b,c\in \Z$. 
	A CM point is the associated root of
	$Q(x,1)$ and is given by 
	\bee
	z_Q=\frac{-b+i\sqrt{|d|}}{2a}\in \HH,
	\eee
	where $\HH$ is the upper-half complex plane, $(a,b,c)=1$ and $d=b^2-4ac<0$. We also denote 
	\be
	z_d = \begin{cases}
		\frac i2\sqrt{|d|} &\quad \text{ if } d\equiv 0\Mod 4,\\
		-\frac12+\frac i2\sqrt{|d|} &\quad \text{ if } d\equiv 1\Mod 4.
	\end{cases}
	\ee
	Define $j_m$ as the unique modular function in 
	$\C[j]$
	of the form 
	$j_m = q^{-m} + O(q)$. 
	It is known that the singular moduli $j_1(z_d)$ are algebraic integers of degree $h(d)$ and that the algebraic trace $\Tr j_1(z_d)$ is an integer.
	Duke \cite{DukeUniformDistrCM} confirmed the following conjecture of Bruinier, Jenkins and Ono \cite{BJOHilbertcls06}:
	\be
	\label{eq:limit}
	\lim_{d\rightarrow -\infty}\frac1{h(d)}\left(\Tr j_1(z_d)-\sum_{z_Q\in\mathcal R(1)}e(-z_Q)\right)=-24,
	\ee
	where $\mathcal R(Y)$ is the rectangle
	\bee
	\mathcal R(Y)=\left\{z=x+iy\in\HH: -\tfrac12\leq x<\tfrac12
	\text{ and }
	y>Y\right\}. 
	\eee
	
	Andersen and Duke \cite{ADasymptotic} generalized this result to sums twisted by genus characters. If we factor a negative discriminant $D=dd'$ as product of discriminants where $d$ is fundamental, then we
	can define a character by
	\begin{equation*}
		\chi_d(Q)=\left\{\begin{array}{ll}
			(\frac dn)& \text{if }(a,b,c,d)=1,\ Q(x,y)=n \text{ for some } x, y\in \Z, \text{ and }(d,n)=1,\\
			0& \text{if }  (a,b,c,d)>1.
		\end{array}
		\right. 
	\end{equation*}
	The twisted trace of $j_m$ is defined by
	\bee
	\Tr_d j_m(z_D) = \sum_{z_Q\in \mathcal F} \chi_d(Q) \omega_Q^{-1} j_m(z_Q),
	\eee
	where 
	\bee
	\mathcal F = \left\{ z\in \HH : -\tfrac 12 \leq \re(z) \leq 0 \text{ and } |z|\geq 1 \right\} \cup \left\{ z\in \HH : 0 < \re(z) < \tfrac 12 \text{ and } |z| > 1 \right\}
	\eee
	is the usual fundamental domain for the action of $\PSL_2(\Z)$ on $\HH$
	and 
	$Q$ runs over all positive definite integral binary quadratic forms of discriminant $d$ with $z_Q\in \mathcal F$.
	Here $\omega_Q=1$ unless $Q$ is $\PSL_2(\Z)$-equivalent to $[a,0,a]$ or $[a,a,a]$, in which case it equals $2$ or $3$, respectively. Let $\delta_1=1$ and $\delta_d=0$ when $d\neq 1$. 
	Andersen and Duke proved the following theorem
	which 
	improves and generalizes \eqref{eq:limit}.
	\begin{theorem}[{\cite[Theorem~2.2]{ADasymptotic}}]
		\label{ADThm2}
		For each negative discriminant $D$, let $d$ be any fundamental discriminant dividing $D$ and let $m\geq 1$.
		Then for $0<Y\ll\frac1m$ we have
		\begin{multline}
			\label{eq:ADThm2}
			\Tr_d j_m(z_D) - \sum_{z_Q\in\mathcal{R}(Y)}\chi_d(Q)(e(-mz_Q)	-e(-m\overline{z_Q}))\\
			=-24\delta_d\sigma_1(m)h(D)
			+O\left(m|D|^{\frac13}Y^{\frac12}
			\left(Y^{-\frac16}+|D|^{\frac{5}{36}}
			m^{\frac13(1+\theta)}\right)(m|D|/Y)^\varepsilon\right),
		\end{multline}
		where $\theta \in [0, \frac 7{64}]$ is
		an admissible exponent toward the Ramanujan conjecture for Maass cusp forms of integral weight.
	\end{theorem}
	For any $\delta > 0$ we have
	$\tau(c) \ll_{\delta} c^{\delta}$, where 
	$\tau(c)$ is the number of divisors of $c$. Let 
	$\ell(\delta)$  be a constant such that for all $c\in\mathbb{N}$ we  have
	\be
	\label{eq:ell}
	\tau(c) \leq \ell(\delta) c^{\delta}.
	\ee 
	For numerical values we can take $\ell(\frac14)=8.447$ and $\ell(\frac 15)=28.117$ \cite[(6)]{NicolasCompositeNumbers}. When $n<0$, for simplicity we write $\tau(n)\defeq \tau(|n|)$ among the paper. 
	
	In Theorem~\ref{Thm2} 
	we 
	provide an effective bound for the error term
	in \eqref{eq:ADThm2}.
	Our methods rely on the estimate \eqref{SumOfPlusKlSum}
	for the sums of plus space Kloosterman sums defined at \eqref{eq:S+12}.  
	\begin{theorem}\label{Thm2}
		Let $D$, $d$ and $m$ be as in Theorem~\ref{ADThm2}, and let 
		$\ell(\delta)$ be as in \eqref{eq:ell}.
		Then we have
		\begin{multline}
			\la \Tr_d j_m(z_D)+24\delta_d\sigma_1(m)h(D)-\sum_{z_Q\in\mathcal{R}(Y)}\chi_d(Q)(e(-mz_Q)-e(-m\overline{z_Q}))\ra\\
			\leq  |D|^{\frac {13}{12}+\frac\delta2} m^{\frac 32}\tau(D)\tau(m)
			Y^{\frac13-\delta}
			\zeta^2(1+\delta)\ell(\delta)\log \frac{2|D|^{\frac12}}Y |\log \delta|\cdot\left\{\begin{array}{ll}
			    106954 & d<0, \\
			    24957m & d>0,
			\end{array}\right.
		\end{multline}
		for any $\delta\in(0,\frac14]$ and $0<Y\leq \frac1{2\pi m}$. 
	\end{theorem}

	We also have the following result, which allows for easier direct comparison with 
	\eqref{eq:limit} and \cite[Thm. 2.1]{ADasymptotic}.
	\begin{theorem}\label{Thm1}
		Let $D$, $d$ and $m$ be as in Theorem~\ref{ADThm2}, and let 
		$\ell(\delta)$ be as in \eqref{eq:ell}. Then 
		\begin{multline}
			\la \Tr_d j_m(z_D)+24\delta_d\sigma_1(m)h(D)-\sum_{z_Q\in\mathcal{R}(\frac1m)}\chi_d(Q)e(-mz_Q)\ra\\
			\leq 
			|D|^{\frac{13}{12}+\frac\delta2} m^{\frac76+\delta} \tau(D) \tau(m)
			\zeta^2(1+\delta)\ell(\delta)\log(2m|D|^{\frac12}) |\log \delta|\cdot \left\{\begin{array}{ll}
			     8.5\times 10^6& d<0, \\
			     2\times 10^6m& d>0,
			\end{array}\right.
		\end{multline}
		for any $\delta\in(0,\frac14]$. 
	\end{theorem}

 \begin{remark}
One may observe that the exponents of $m$ and $|D|$ in Theorem~\ref{Thm2} and Theorem~\ref{Thm1} are larger than those in Theorem~\ref{ADThm2} and \cite[Thm. 2.1]{ADasymptotic}, respectively. This is expected since our method is different and aims to obtain explicit constants. 

To be precise, since $h(D)>c_{\ep}|D|^{\frac12-\ep}$ \cite{Siegel1935} and $\limsup \frac{\sigma_1(m)}{m\log\log m}=e^\gamma$, when $d=1$ our bound $|D|^{\frac{13}{12}+\ep}m^{\frac32+\ep}$ grows faster than that term, while Andersen and Duke's result $|D|^{\frac{17}{36}+\ep}m^{\frac 56+\frac\theta 3+\ep}$ is practical. When $d\neq 1$, the sum on $e(-mz_Q)$ grows exponentially and any polynomial bound is sufficient.  However, the benefit of our method is the explicit choice of $Y$ in Corollary~\ref{Coro3}. 
\end{remark}

	By choosing $Y$ small enough that the error term 
	in \eqref{eq:ADThm2}
	tends to zero,
	Andersen and Duke
	obtain the following corollary:
	\begin{corollary}[\cite{ADasymptotic}, Corollary~2.3]
		Let $D$, $d$ and $m$ be as in Theorem~\ref{ADThm2} and let $Y = Cm^{-A} |D|^{-B}$ , where $A > 3$, $B > 1$, and $C > 0$ are constants. Then $\Tr_d j_m (z_D )$ is the nearest integer to
		\[-24\delta_d\sigma_1(m)h(D)+\sum_{z_Q\in\mathcal{R}(Y)}\chi_d(Q)(e(-mz_Q)-e(-m\overline{z_Q}))\]
		provided $m|D|$ is sufficiently large compared to $C$. 
	\end{corollary}
	Choosing $\delta=\frac15$ and $\tau(n)\leq \ell(\frac15)n^{\frac 15}$
	in Theorem~\ref{Thm2}, we can obtain a similar result with 
	explicit constants $A$, $B$, and $C$ from rough estimates.
	\begin{corollary}\label{Coro3}
		Let $D$, $d$ and $m$ be as in Theorem~\ref{ADThm2}. Then $\Tr_d j_m (z_D )$ is the nearest integer to
		\bee
		-24\delta_d\sigma_1(m)h(D)+\sum_{z_Q\in\mathcal{R}(Y)}\chi_d(Q)(e(-mz_Q)-e(-m\overline{z_Q}))
		\eee
		provided that
		\bee
		Y<10^{-100} |D|^{-11}\cdot \left\{
		\begin{array}{lr}
		     m^{-13}& d<0, \\
		     m^{-21}& d>0.
		\end{array}
		\right.
		\eee
	\end{corollary}

	In the next section we give some background material. In Section \ref{sec:projectedSeries}
	we provide expressions for the necessary 
	Fourier coefficients of Poincar\'e series of half-integral weight in the Kohnen plus space.
	In Section
	\ref{sec:zeta}
	we obtain estimates for the Kloosterman-Selberg zeta function.
	In Section
	\ref{sec:perron}  we apply Perron's formula and obtain an effective bound 
	for the sums of plus space Kloosterman sums. 
	In the last section we 
	complete the proofs of Theorem~\ref{Thm2} and
	Corollary~\ref{Coro3}.

	\section{Preliminaries}
	Let $\Gamma = \Gamma_0(4)$.
	We say that $\nu:\Gamma\to \C^\times$ is a multiplier system of weight $k\in \mathbb{R}$ if
	\begin{enumerate}[label=(\roman*)]
		\item $|\nu|=1$,
		\item $\nu(-I)=e^{-\pi i k}$, and
		\item $\nu(\gamma_1 \gamma_2) =w(\gamma_1,\gamma_2)\nu(\gamma_1)\nu(\gamma_2)$ for all $\gamma_1,\gamma_2\in \Gamma$, where
		\[w(\gamma_1,\gamma_2)\defeq j(\gamma_2,\tau)^k j(\gamma_1,\gamma_2\tau)^k j(\gamma_1\gamma_2,\tau)^{-k} \text{ \ and \ } j(\gamma,\tau)\defeq\frac{cz+d}{|cz+d|}. \]
	\end{enumerate}
	If $\nu$ is a multiplier system of weight $k$, then it is also a multiplier system of weight $k'$ for any $k'\equiv k\Mod 2$, and the conjugate $\overline\nu$ is a multiplier system of weight $-k$.

	For $\gamma\in \SL_2(\R)$ we define the weight $k$ slash operator
	by
	\be
	f\sl_k \gamma := j(\gamma,z)^{-k} f(\gamma z),
	\ee
	and the weight $k$ hyperbolic Laplacian by
	\bee
	\Delta_k := y^2 \bigg( \frac{\partial^2}{\partial x^2} + \frac{\partial^2}{\partial y^2} \bigg) - iky \frac{\partial}{\partial x}.
	\eee
	A real analytic function $f:\HH\to\C$ is an eigenfunction of $\Delta_k$ with eigenvalue $\lambda$ if
	\be
	\label{eq:eigen}
	\Delta_k f + \lambda f = 0.
	\ee
	For such a $f$, we write
	\bee
	\lambda = \tfrac 14 + r^2,
	\eee
	and refer to $r$ as the spectral parameter of $f$.
	A function $f:\HH\to\C$ is automorphic of weight $k$ and multiplier $\nu$ for $\Gamma$ if
	\be
	\label{eq:slashtransformation}
	f \sl_k \gamma = \nu(\gamma) f \qquad \text{ for all }\gamma\in\Gamma.
	\ee
	Let $\mathcal{A}_k(N,\nu)$ denote the space of all such functions.
	A smooth automorphic function which is also an eigenfunction of $\Delta_k$ and which has at most polynomial growth at the cusps of $\Gamma$
	is called a Maass form.
	Let $\mathcal{L}_k(\nu)$ 
	denote the $L^2$-space of automorphic functions with respect to the Petersson inner product
	\bee
	\label{eq:petersson}
	\langle f,g \rangle := \int_{\Gamma\backslash\HH} f(z) \overline{g(z)} \, \mathrm{d}\mu, \qquad \mathrm{d}\mu := \frac{\mathrm{d}x\,\mathrm{d}y}{y^2}.
	\eee
	The spectrum of $\Delta_k$ is real and contained in $[\lambda_0(k),\infty)$, where $\lambda_0(k) := \frac {|k|}2 (1-\frac{|k|}2)$.

	Let $\mathcal{V}_k(\nu)$ denote the orthogonal complement in $\mathcal{L}_k(\nu)$ of the space generated by Eisenstein series.
	The spectrum of $\Delta_k$ on $\mathcal{V}_k(\nu)$ is countable and of finite multiplicity.
	Eigenfunctions of $\Delta_k$ in $\mathcal{V}_k(\nu)$ are called Maass cusp forms.
	Let $\{f_j\}$ be an orthonormal 
	basis of $\mathcal{V}_k(\nu)$, and for each $j$ let $\lambda_j = \frac 14+r_j^2$
	denote the  Laplace eigenvalue.

	We are interested in the multiplier system
	$\nu_{\theta}$ of weight 
	$\frac{1}{2}$ on $\Gamma$
	defined by
	\be
	\theta(\gamma z) = \nu_{\theta}(\gamma) \sqrt{cz+d}\; \theta(z)
	\qquad \text{for } \gamma = \begin{psmallmatrix}
		a&b\\c&d
	\end{psmallmatrix}
	\in
	\Gamma,
	\ee
	where we denote $\(\frac \cdot\cdot\)$ as the extended Kronecker symbol and
\begin{equation}
\theta(z) \defeq \sum_{n\in\Z} e(n^2 z), \quad \nu_{\theta}(\gamma)=\(\frac cd\)\ep_d^{-1}, \quad \ep_d=\left\{ \begin{array}{ll}
	1&d\equiv 1\Mod 4,\\
	i&d\equiv 3\Mod 4. 
\end{array}\right.
\end{equation}

	Let $k=\pm \frac12$ and let 
	$c$ be a positive integer divisible by $4$.
	For integers
	$m, n$ satisfying
	$(-1)^{k-\frac12} n \equiv 0,1 \Mod{4}$
	and
	$(-1)^{k-\frac12} m \equiv 0,1 \Mod{4}$,
	define
	(\cite[(1.9)]{ADinvariants})
	\be
	\label{eq:skplus}
	S_{k}^{+}(m,n,c) \defeq e\( -\frac{k}{4}\)
	\sum_{d\Mod{c}} \(\frac{c}{d}\) \varepsilon_{d}^{2k} 
	e\(\frac{m\bar{d} +nd}{c} \) \times \begin{cases}
		1 \quad &\text{ if } 8 \mid c,\\
		2 \quad &\text{ if } 4 \mid\mid c,
	\end{cases}
	\ee
	where $d \bar{d} \equiv 1\Mod{c}$.
	We have the following equality for these Kloosterman sums: 
	\be
	\label{eq:KloostermanRel}
	S_{k}^{+}(m,n,c) = S_{k}^{+}(n,m,c) = S_{-k}^{+}(-m,-n,c)\in \R.
	\ee
	They also have the Weil-type bound
	(\cite[(1.12)]{ADinvariants})
	\be \label{WeilBd}
	|S_{k}^{+}(m,n,c)|\leq 2\tau(c) (m,n,c)^{\frac12} c^{\frac12}.
	\ee

	We factor a negative discriminant $D<0$ as $D=dd'$ where $d$ is a fundamental discriminant and $d'$ is a discriminant. Kohnen's identity (\cite[Proposition~5]{Kohnen85} and \cite[Lemma~8]{dit}) 
	shows that quadratic Weyl sums are finite sums of Kloosterman sums. Specifically, we have 
	\begin{align}
		\label{eq:weylKloosterman}
		\nonumber
		T_m(d,d';c)&\defeq\sum_{\substack{b\Mod c\\b^2\equiv D(c)}}\chi_d\(\left[\frac c4,b,\frac{b^2-D}c\right]\)e\(\frac{2mb}c\)\\
		&=\sum_{n|(m,\frac c4)}\(\frac dn\)\sqrt{\frac{2n}{c}}\,S_{\frac12}^+\(d',\frac{m^2}{n^2}d,\frac cn\).
	\end{align}
	The following lemma relates 
	$\Tr_d j_m(z_D)$
	to Weyl sums.
	\begin{lemma}
		\label{lem:TrWeyl}
		Let $D = dd'$ be a negative discriminant and $d$ a fundamental discriminant.
		For $Y>0$ we have
		\begin{multline*}
			\Tr_d j_m(z_D)+24\delta_d\sigma_1(m)h(D)-\sum_{z_Q\in\mathcal{R}(Y)}\chi_d(Q)e(-mz_Q)\\
			=\!\!\!\sum_{4|c\geq \frac{2|D|^{1/2}}Y}\!\!\!T_m(d,d';c)\sinh \(\frac{4\pi m}c |D|^{\frac12}\)
			-\frac12 \!\!\!\sum_{4|c<\frac{2|D|^{1/2}}Y} \!\!\!T_m(d,d';c) \exp\(-\frac{4\pi m}c |D|^{\frac12}\)
		\end{multline*}
		and
		\begin{align*}
			\nonumber
			\Tr_d j_m(z_D)+24\delta_d\sigma_1(m)h(D)&-\sum_{z_Q\in\mathcal{R}(Y)}\big(\chi_d(Q)e(-mz_Q)-\chi_d(Q) e(-m\overline{z_Q})\big)\\
			&=\sum_{4|c\geq \frac{2|D|^{1/2}}Y} T_m(d,d';c)\sinh \(\frac{4\pi m}c |D|^{\frac12}\). 
		\end{align*}
	\end{lemma}
	\begin{proof}
		This lemma follows from \cite[(3.3)]{ADasymptotic}
		and \cite[Lemma~3.2]{ADasymptotic}
	\end{proof}

	\section{
		Fourier coefficients of projected Poincar\'e series}
	\label{sec:projectedSeries}
Let $\Gamma=\Gamma_0(4)$ and recall our weight $\frac12$ multiplier $\nu_\theta$ on $\Gamma$. 
	For any cusp $\mathfrak a$ of $\Gamma$, its scaling matrix $\sigma_{\mathfrak{a}}\in\GL_2^+(\R)$ satisfies 
	\[\sigma_{\mathfrak{a}}\infty=\mathfrak{a} \text{ \ and \ } \sigma_{\mathfrak{a}}^{-1}\Gamma_{\mathfrak{a}}\sigma_{\mathfrak{a}}=\Gamma_\infty,\]
	where $\Gamma_{\mathfrak{a}}$ is the stabilizer of $\mathfrak{a}$ and in particular $\Gamma_\infty=\{\pm\begin{psmallmatrix}
		1&n\\0&1
	\end{psmallmatrix}:n\in \Z\}$. We define $\alpha_{\mathfrak{a},\nu_\theta}\in[0,1)$ by
	\[\nu_\theta\(\sigma_{\mathfrak{a}}\begin{pmatrix}
		1&1\\0&1
	\end{pmatrix}\sigma_{\mathfrak{a}}^{-1}\)=e\(-\alpha_{\mathfrak{a},\nu_\theta}\).\]
	For $m\in \Z$, define $m_\mathfrak{a}\defeq m-\alpha_{\mathfrak{a},\nu_\theta}$. One can check that $m_{\infty}=m$. 
	The Kloosterman sums for a cusp pair $(\mathfrak{a},\mathfrak{b})$ is given by
	\be
	\label{eq:KlooSumAB}
	S_{\mathfrak{a},\mathfrak{b}}(m,n,c,\nu)=\sum_{\gamma=\begin{psmallmatrix}
			a&b\\c&d
		\end{psmallmatrix}\in \Gamma_\infty\setminus\sigma_{\mathfrak{a}}^{-1}\Gamma
		\sigma_{\mathfrak{b}}/\Gamma_\infty}
	\overline{{\nu}_{\mathfrak{a}\mathfrak{b}}}\begin{pmatrix}
		a&b\\c&d
	\end{pmatrix}
	e\left(\frac{m_{\mathfrak a}a+n_\mathfrak b d}{c}\right), 
	\ee
	where 
	\bee
	\nu_{\mathfrak{a}\mathfrak{b}}(\gamma)
	=\nu_\theta(\sigma_{\mathfrak{a}}\gamma\sigma_{\mathfrak{b}}^{-1})
	\dfrac{w(\sigma_{\mathfrak{a}}\gamma\sigma_{\mathfrak{b}}^{-1},\sigma_{\mathfrak{b}})}    {w(\sigma_{\mathfrak{a}},\gamma)}.
	\eee
	If $f \in  \mathcal V_{\frac12}(\nu_\theta)$, the space of Maass cusp forms,
	then for each cusp $\mathfrak a$ we have
	\be
	\label{eq:fouriercusp}
	(f|_{\frac12}\sigma_{\mathfrak{a}})(z) = \sum_{n\in \Z}c_{f,\mathfrak{a}}(n,y)e(n_{\mathfrak a}x).
	\ee
	This space decomposes as
	$ \mathcal V_{\frac12}(\nu_\theta) =  \mathcal V^{+}_{\frac12}(\nu_\theta)
	\oplus
	\mathcal V^{-}_{\frac12}(\nu_\theta)$, where for $f \in  \mathcal V_{\frac12}(\nu_\theta)$,
	we have 
	${f \in  \mathcal V^+_{\frac12}(\nu_\theta)}$
	if and only if
	$c_{f,\infty}(n,y) =0$
	for 
	$n \equiv 2,3 \Mod{4}$.
	
	There are three nonequivalent cusps of $\Gamma_0(4)$ represented by
	$\infty$, $0$ and $\frac12$. We choose their respective scaling matrices as
	\bee
	\sigma_\infty=I,\qquad \sigma_0=\begin{pmatrix}
		0&-\frac12\\2&0
	\end{pmatrix},\qquad \sigma_{\frac12}=\begin{pmatrix}
		1&-\frac12\\2&0
	\end{pmatrix} 
	\eee
	with
	\bee
	\alpha_{\infty, \nu_\theta}=0, \qquad \alpha_{0,\nu_\theta}=0, \qquad 
	\alpha_{\frac12,\nu_\theta}= \tfrac{3}{4}.
	\eee
	Note that 
	\be
	\label{eq:S+12}
	S_{\frac12}^+(m,n,c)=
	e(-\tfrac 18)S_{\infty,\infty}(m,n,c,\nu_\theta)
	\times
	\begin{cases}
		1  &\text{ if } 8\mid c,\\
		2  &\text{ if } 4\mid\mid c.
	\end{cases}
	\ee

	Following the notation in 
	\cite[\S 5]{ADinvariants},
	we define Kohnen's operator $L$ on 
	automorphic functions as follows. 
	If $f$ satisfies 
	\eqref{eq:slashtransformation} with $\Gamma = \Gamma_0(4)$, then
	\be
	Lf \defeq \frac{1}{2(1+i)} \sum_{w=0}^{3}
	f\sl_{\frac12} \begin{pmatrix}
		1+w & \tfrac14\\
		4w & 1
	\end{pmatrix}.
	\ee
	Kohnen's operator $L$ commutes with the weight $\frac12$ slash operator and is self-adjoint with respect to the Petersson inner product. It also satisfies $(L-1)(L+\frac12)=0$. 
	Let $\pr^+\defeq\frac23\(L+\frac12\)$. 
	Then  $(\pr^+)^2=\pr^+$ and we have a projection operator 
	\bee
	\pr^+:
	\mathcal{V}_{\frac12}(\nu_\theta)
	\rightarrow \mathcal{V}^+_{\frac12}(\nu_\theta).
	\eee

\begin{lemma}[\cite{ADinvariants}, Lemmas 5.1 and 5.5]
\label{lem:ADlemmas_cfa}
Suppose that 
$f \sl_{\frac12} \gamma = \nu_\theta(\gamma) f $  for all $\gamma\in\Gamma$.
For each 
cusp $\mathfrak{a}$ of 
$\Gamma$
write the Fourier expansion of $f$ as
in \eqref{eq:fouriercusp}.
Then 
\bee
\label{eq:cL}
c_{(L+\frac12)f,\infty}(n,y)=
\begin{cases}         
c_{f,\infty}(n,y)+\frac1{2(1-i)} c_{f,0}\left(\frac n4,4y\right) &\text{ if } n\equiv 0\Mod{4},\\
c_{f,\infty}(n,y)+\frac1{2(1-i)} c_{f,\frac12}\left(\frac {n+3}4,4y\right) &\text{ if } n\equiv 1\Mod{4},\\
0&\text{ otherwise}. 
\end{cases}
\eee
Moreover, for $4\mid \mid c$ we have
\bee
S_{\infty,\infty}(m,n,c,\nu_\theta)=(1+i)\times 
\begin{cases}	S_{\infty,0}\left(m,\frac n4,\frac c2,\nu_\theta\right)
	&\text{ if } n\equiv 0\Mod 4, \\
	S_{\infty,\frac12}\left(m,\frac {n+3}4,\frac c2,\nu_\theta\right)
	&\text{ if } n\equiv 1\Mod 4. 
\end{cases}
\eee
\end{lemma}

For 
$m>0$ and 
$(k,\nu) = (\frac 12, \nu_{\theta})$ or $(-\frac 12, \overline{\nu_{\theta}})$, define
the Poincar\'e series $P_m(z,s,k,\nu) \in 
\mathcal{V}_{k}(\nu)$ by
\be
P_m(z,s,k,\nu)=\sum_{\gamma=\begin{psmallmatrix}
		a&b\\c&d
	\end{psmallmatrix}\in\Gamma_\infty\setminus\Gamma}\overline{\nu(\gamma)} j(\gamma,z)^{-k} \im(\gamma z)^s e(m\gamma z), \qquad \re s>1.
\ee
In the following lemma we write
$c_{(L+\frac12)P_m(z,s,\frac12,\nu_\theta),\infty}(n,y)$
as a sum
involving 
Kloosterman sums.
\begin{lemma}\label{FourCoeffOfPoincProjToPlusSpace}
	For integers
	$m, n$ satisfying
	$m, n \equiv 0,1 \Mod{4}$,
	and $m>0$
	we have
	\bee
	c_{(L+\frac12)P_m(z,s,\frac12,\nu_\theta),\infty}(n,y)=y^s e(\tfrac 18)\sum_{4|c>0}\frac{S_{\frac12}^+(m,n,c)}{c^{2s}}B(c,m,n,y,s,\tfrac12)
	+\delta_{m,n}y^se^{-2\pi my}, 
	\eee
	where 
	\be
	\label{eq:Bcmnysk}
	B(c,m,n,y,s,k) = 
	y\int_\R 
	\left(\frac{u+i}{|u+i|}\right)^{-k}
	e\left(-\frac{m}{c^2y(u+i)}-nuy\right)
	\frac{\mathrm{d}u}{y^{2s}(u^2+1)^{s}}.
	\ee
\end{lemma}
\begin{proof} 
	To use 
	Lemma~\ref{lem:ADlemmas_cfa} 
	we need to find $c_{f,\mathfrak{a}}$ for all three cusps $\mathfrak{a}$ in order to compute $c_{(L+\frac12)f,\infty}$. 
	We have the following expressions for Poincar\'e series at the cusps $\infty$, $0$, and $\frac12$
	(see \cite[Lemma~3.4]{di} and \cite[(15)]{Proskurin2005} for details):
	\bee
	\label{eq:poincareinf}
	P_m(z,s,\tfrac12,\nu_\theta)=y^se(mz)+y^s\sum_{n\in \Z}e(nx)\sum_{4|c>0}\frac{S_{\infty,\infty}(m,n,c,\nu_\theta)}{c^{2s}}B(c,m,n,y,s,\tfrac12), 
	\eee
	\bee
	\label{eq:poincarezero}
	(P_m(z,s,\tfrac12,\nu_\theta)|_{\frac12}\sigma_0)(z,s)
	=y^s\sum_{n\in \Z}e(nx)
	\sum_{2\mid \mid c>0}
	\frac{S_{\infty,0}(m,n,c,\nu_\theta)}{c^{2s}}
	B(c,m,n,y,s,\tfrac12), 
	\eee
	and
	\bee
	\label{eq:poincarehalf}
	(P_m(z,s,\tfrac12,\nu_\theta)|_{\frac12}\sigma_{\frac12})(z,s)=y^s
	\sum_{n\in \Z}e(n_{\frac12}x)
	\sum_{2\mid \mid c>0}
	\frac{S_{\infty,\frac12}(m,n,c,\nu_\theta)}{c^{2s}}B(c,m,n_{\frac12},y,s,\tfrac12). 
	\eee
	Thus, 
	\bee
	c_{P_m(z,s,\frac12,\nu_\theta),\infty}(n,y)=y^s\sum_{4|c>0}\frac{S_{\infty,\infty}(m,n,c,\nu_\theta)}{c^{2s}}B(c,m,n,y,s,\tfrac12)
	+\delta_{m,n}y^se^{-2\pi my}.
	\eee
	For $n\equiv 0 \Mod{4}$ we see that
	\be
	\label{eq:coeff0}
	c_{P_m(z,s,\frac12,\nu_\theta),0}\left(\tfrac n4,4y\right)
	=
	(4y)^s\sum_{2\mid \mid c>0}
	\frac{S_{\infty,0}
		(m,\tfrac n4,c,\nu_\theta)}{c^{2s}}B(c,m,\tfrac n4,4y,s,\tfrac12),
	\ee
	and for $n\equiv 1 \Mod{4}$ we obtain
	\be
	\label{eq:coeffhalf}
	c_{P_m(z,s,\frac12,\nu_\theta),\frac12}\left(\tfrac {n+3}4,4y\right)
	=
	(4y)^s\sum_{2\mid \mid c>0}
	\frac{S_{\infty,\frac12}(m,\tfrac{n+3}{4},c,\nu_{\theta})}{c^{2s}}B(c,m,\tfrac{n}{4},4y,s,\tfrac12).
	\ee
	From \eqref{eq:Bcmnysk}
	we can see that
	$B(\frac c2,m,\frac n4,4y,s,k)=4^{1-2s}B(c,m,n,y,s,k)$.
	Thus, we can rewrite 
	\eqref{eq:coeff0}
	and
	\eqref{eq:coeffhalf}
	using 
	Lemma~\ref{lem:ADlemmas_cfa}
	as follows.
	For $n \equiv 0 \Mod{4}$ we have
\begin{align}
	\label{eq:cpm0simplified}
	\nonumber
	c_{P_m(z,s,\frac12,\nu_\theta),0}\left(\tfrac n4,4y\right)
	&=
	(4y)^s\sum_{4\mid \mid c>0}
	\frac{S_{\infty,0}
		(m,\tfrac n4,\tfrac c2,\nu_\theta)}{(c/2)^{2s}}B(\tfrac c2,m,\tfrac n4,4y,s,\tfrac12)\\
	&=
	\frac{4y^s}{1+i}\sum_{4\mid \mid c>0}
	\frac{S_{\infty,\infty}
		(m,n,c,\nu_{\theta})}{c^{2s}}B(c,m,n,y,s,\tfrac12).
	\end{align}
	Similarly, for $n \equiv 1 \Mod{4}$ we have
	\be
	\label{eq:cpmhalfsimplified}
	c_{P_m(z,s,\frac12,\nu_\theta),\frac12}\left(\tfrac {n+3}4,4y\right)
	=
	\frac{4y^s}{1+i}\sum_{4\mid \mid c>0}
	\frac{S_{\infty,\infty}(m,n,c,\nu_\theta)}{c^{2s}}B(c,m,n,y,s,\tfrac12).
	\ee
	The result now follows from 
	Lemma~\ref{lem:ADlemmas_cfa}
	and
	\eqref{eq:S+12}.
	\qedhere
\end{proof}

\section{Estimates for the Kloosterman-Selberg zeta function}
\label{sec:zeta}
In this section we will give a bound for the 
Selberg-Kloosterman zeta function $Z_{m,n}^+$ defined by
\be
Z_{m,n}^+(s)\defeq \sum_{4|c>0} \frac{S_{\frac12}^+(m,n,c)}{c^{2s}},\qquad s=\sigma+it,\quad \sigma =\re(s)>1
\ee
and with a meromorphic continuation to $\C$.  
The computations 
are similar to those in 
\cite[\S 4 - \S 6]{Oscar20},
so we omit some details.
Let
\be
I_{m,n}(s_1,s_2)\defeq \left\langle \pr^+ P_m\left(z,s_1,\tfrac12,\nu_\theta\right), \overline{P_{-n}\left(z,s_2,-\tfrac12,\overline{\nu_\theta}\right)}\right\rangle,
\ee
where  
$m, n \equiv 0,1\Mod 4$,
$m>0$, $n<0$, 
and 
$\re(s_2)>\re(s_1)>1$. Define
\begin{multline*}
	Q_{m,n}^+(s_1,s_2)\defeq 
	Z_{m,n}^+(s_1)
	\int_{-\infty}^\infty \left(\frac{u+i}{|u+i|}\right)^{-\frac12}\frac{1}{(u^2+1)^{s_1}}\\
	\times
	\(
	\int_0^\infty y^{s_2-s_1-1}e\left(-ny(u+i)\right)\(e\(\frac{-m}{c^2y(u+i)}\)-1\) \mathrm{d}y\) \mathrm{d}u.
\end{multline*}
As a result of Lemma~\ref{FourCoeffOfPoincProjToPlusSpace}, since $\pr^+=\frac23(L+\frac12)$, 
we obtain the following expression for the inner product of two Poincaré series by the unfolding method (see, e.g., \cite[Lemma~3.1]{Oscar20} 
or
\cite[page 452]{pribitkin}):
\begin{align}
	\label{Imn}
	\nonumber
	I_{m,n}(s_1,s_2)&=\frac{2\Gamma(s_2+s_1-1)\Gamma(s_2-s_1)}{3\Gamma(s_2+\frac 14)\Gamma(s_1-\frac 14)}  Z_{m,n}^+(s_1) \pi^{s_1-s_2+1} |n|^{s_1-s_2}  4^{1-s_2}\\
	&+  \frac23\;e\(\frac 18\)Q_{m,n}^+(s_1,s_2).
\end{align}
Next we need a bound for 
$I_{m,n}(s,s+2)$.

\begin{proposition}
	\label{prop:innerBound}
	Suppose that $m, n \equiv 0,1\Mod 4$, $m>0$ 
	and $n<0$.
	Let
	$s=\sigma+it$ with $\sigma=\frac12+\frac{\delta}2$ and $\delta\in(0,\frac14]$. Then we have
	\bee
	|I_{m,n}(s,s+2)|\leq 7.42\zeta^2(1+\delta)\tau(m)^{\frac{1}{2}}\tau(n)^{\frac{1}{2}}
	m^{\frac34}|n|^{-\frac 74}.
	\eee
\end{proposition}
\begin{proof}
	From \cite[(2.4)]{gs} and \cite[(A.2.9)]{sarnakapp} we see that
	\begin{align*}
	|I_{m,n}&(s,s+2)|
	=|\langle    \pr^+ P_m\left(z,s,\tfrac12,\nu_\theta\right), \overline{P_{-n}\left(z,s+2,-\tfrac12,\overline{\nu_\theta}\right)} \rangle| \\
	&\leq 4\pi m |s-\tfrac 14|\,  \lVert R_{s(1-s)} \rVert \, \lVert \pr^+ P_m(z,s+1,\tfrac12,\nu_\theta) \rVert  \,  
	\lVert P_{-n}(z,s+2,-\tfrac12,\overline{\nu_\theta}) \rVert\\
	&\leq \frac{ 4\pi m |s-\tfrac 14|\,\lVert \pr^+ P_m(z,s+1,\tfrac12,\nu_\theta) \rVert  \,
		\lVert P_{-n}(z,s+2,-\tfrac12,\overline{\nu_\theta}) \rVert }{\text{distance}(s(1-s),\text{spectrum}(\Delta_{\frac12}))}  ,
\end{align*}
	where 
	$R_{s(1-s)} = (\Delta_{\frac12} + s(1-s))^{-1}$
	is the resolvent of 
	$\Delta_{\frac12}$.

	For $|t|>1$ we have the following inequalities: $|s-\frac 14| \leq 1.07 |t|$, 
	\bee
	\text{dist}(s(1-s),\text{spectrum}(\Delta_{k}))\geq|t(2\sigma-1)| \qquad \text{(\cite[before Lemma~2]{gs})},
	\eee
	and
	$\frac1{\delta}\leq \zeta(1+\delta)$, so
	\be
	\label{eq:largetImn}
	|I_{m,n}(s,s+2)|
	\leq 13.45m \zeta(1+\delta) 
	\lVert  \pr^+ P_m\left(z,s+1,\tfrac12,\nu_\theta\right) 
	\rVert  \,  
	\lVert P_{-n}\left(z,s+2,-\tfrac12,\overline{\nu_\theta}\right) \rVert.
	\ee
	
	For $|t|\leq 1$, we see that $\re(s(1-s))\leq \frac 54$.
	For the discrete spectrum of $\Delta_{\frac12}$ on $\Gamma_0(4)$, $\lambda_0(\frac12)$ is $\frac 3{16}$ with eigenfunction $y^{\frac14}\theta(z)$. We claim that $\lambda_1(\frac12)>2.4$ as the first eigenvalue larger than $\lambda_0(\frac12)$. This is because for the Maass cusp form corresponding to $\lambda_1=\lambda_1(\frac12)$, by \cite[\S 3]{sarnakAdditive} there is an even weight $0$ Maass form on $\Gamma_0(2)$ with eigenvalue $4\lambda_1-\frac 34$. To obtain the bound, we verify with \cite{lmfdb} for weight $0$ Maass forms on $\Gamma_0(2)$ and get $4\lambda_1-\frac 34>8.92$. 
	
	Thus, for $|t|\leq 1$, we have
	\bee
	\text{distance}(s(1-s),\text{spectrum}(\Delta_{\frac12})) 
	= \la s(1-s)-\tfrac{3}{16} \ra
	= \la s-\tfrac14\ra \la s-\tfrac34 \ra  
	\geq \tfrac 18 |s-\tfrac14|.
	\eee
	Using $\zeta(1+\delta) \geq \zeta(\frac{5}{4}) \geq  4.59$, for these values of $t$ we have
	\begin{equation}
		\label{eq:smalltImn}
		|I_{m,n}(s,s+2)|
		\leq 21.91m
		\zeta(1+\delta) 
		\lVert  \pr^+ P_m\left(z,s+1,\tfrac12,\nu_\theta\right) 
	\rVert  \, 
	\lVert P_{-n}\left(z,s+2,-\tfrac12,\overline{\nu_\theta}\right) \rVert.
	\end{equation} 

Next we estimate the norms of Poincar\'e series. By \eqref{WeilBd} and the proof of \cite[(16.50)]{iwkow} we have
\be\label{TrivEstimate}
\sum_{4|c>0}\frac{\la S_{\frac12}^+(m,n,c)\ra}{c^{\alpha+\frac12}}\leq 6\zeta^2(\alpha)\tau((m,n))\leq 6\zeta^2(\alpha)\tau(m)^{\frac 12}\tau(n)^{\frac 12}\qquad \text{for }\alpha>1,\ee
where the last inequality is because $(m,n)$ divides $m$ and $n$ and $\min(B_1,B_2)\leq \sqrt{B_1B_2}$. 
Similarly we have
\be \label{TrivEstimateInftyInfty}
\sum_{4|c>0}\frac{\la S_{\infty,\infty }^+(m,n,c)\ra}{c^{\alpha+\frac12}}\leq 3\zeta^2(\alpha)\tau(m)^{\frac 12}\tau(n)^{\frac 12}\qquad \text{for }\alpha>1. 
\ee
	Thanks to the fact that $\pr^+$ is a Hermitian operator with respect to the Petersson inner product (see \cite[(2.5)]{BringmannKaneVia13}), we can compute the norm by
	\bee
	\left\|\pr^+P_m\left(z,s+1,\tfrac12,\nu_\theta\right)\right\|^2=\left\langle \pr^+P_m\left(z,s+1,\tfrac12,\nu_\theta\right), P_m\left(z,s+1,\tfrac12,\nu_\theta\right)\right\rangle.
	\eee
	From
	\eqref{TrivEstimate}, unfolding the above inner product and
	arguing as in the proof of
	\cite[Proposition~5.2]{Oscar20} 
	we see that
	\begin{align}\label{Pm_bound}
		\begin{split}
		\big\|\pr^+P_m &\left(z,s+1,\tfrac12,\nu_\theta\right)\big\|^2\\
		&\leq\frac{\Gamma(2\sigma+1)}{(4\pi m)^{2\sigma+1}}+\frac{4}{3} \sum_{4|c>0}
		\frac{\la S_{\frac12}^+(m,m,c)\ra}{c^{2\sigma+2}}
		\int_{-\infty}^\infty
		\frac{K_0\left(\frac{4\pi m}{c(u^2+1)^{1/2}}\right)}
		{(u^2+1)^{\sigma+1}}\mathrm{d}u\\ 
		&\leq \frac{0.007}{m^2}+
		\frac{0.88}{\sqrt m }\sum_{4|c>0}
		\frac{\la S_{\frac12}^+(m,m,c)\ra}{c^{\frac52+\delta}}\\
		&\leq \frac{14.3}{\sqrt m}\tau(m),
	\end{split}
	\end{align}
where in the second step we use the bound $K_0(y)<0.975y^{-\frac12}$ in \cite[Lemma~5.1]{Oscar20} to get
\[\int_{-\infty}^\infty K_0\(\frac{4\pi m}{c(u^2+1)^{\frac12}}\)\frac{\mathrm{d}u}{(u^2+1)^{\sigma+1}}\leq 0.975\(\frac{c}{4\pi m}\)^{\frac12}\frac{\sqrt \pi \Gamma(\frac 34)}{\Gamma(\frac54)}\leq 0.66\(\frac cm\)^{\frac12} \]
and in the last step we apply \eqref{TrivEstimate}. 
	We compute the norm of $P_{-n}$ in a similar way. From  \eqref{eq:KloostermanRel} we have
	\be
	\label{eq:KloostermanConj}
	|S_{\infty,\infty}(-n,-n,c,\overline{\nu_\theta})|
	=
	|S_{\infty,\infty}(n,n,c,\nu_\theta)|
	\ee
	and using \eqref{TrivEstimateInftyInfty} we obtain
	\begin{align}\label{Pn_bound}
		\begin{split}
		\|P_{-n}&\left(z,s+2,-\tfrac12,\overline{\nu_\theta}\right)\|^2\\
		&\leq \frac{\Gamma(2\sigma+3)}{(4\pi |n|)^{2\sigma+3}}+2\sum_{4|c>0}
		\frac{|S_{\infty,\infty}(n,n,c,\nu_\theta)|}{c^{2\sigma+4}}\int_{-\infty}^\infty\frac{K_0\left(\frac{4\pi |n|}{c(u^2+1)^{1/2}}\right)}{(u^2+1)^{\sigma+2}}\mathrm{d}u\\
		&\leq\frac{3}{128\pi^4} |n|^{-4}
		+0.0026 |n|^{-\frac{7}{2}}
		\sum_{4|c>0}
		\frac{|S_{\infty,\infty}(n,n,c,\nu_\theta)|}{c^{\frac 32+\delta}}\\
		&\leq 0.008|n|^{-\frac72}\zeta^2(1+\delta) \tau(n), 
	\end{split}
	\end{align}
	where in the second step we use the bound $K_0(y)<1.7 y^{-\frac 72}$ in \cite[Lemma~5.1]{Oscar20} to get
\[\int_{-\infty}^\infty K_0\(\frac{4\pi |n|}{c(u^2+1)^{\frac12}}\)\frac{\mathrm{d}u}{(u^2+1)^{\sigma+2}}\leq 1.7\(\frac{c}{4\pi |n|}\)^{\frac72}\frac{\sqrt \pi \Gamma(\frac 14)}{\Gamma(\frac34)}\leq 0.0013\(\frac c{|n|}\)^{\frac 72}. \]
	The result then follows from
	\eqref{eq:largetImn}, \eqref{eq:smalltImn},
	\eqref{Pm_bound}, and \eqref{Pn_bound}.
\end{proof}

\begin{theorem}\label{BoundZmn}
	Suppose that $m, n\equiv 0,1\Mod 4$, $m>0$ and $n<0$.
	Let
	$s=\sigma+it$ with $\sigma=\frac12+\frac{\delta}2$ and $\delta\in(0,\frac14]$. Then we have
	\begin{equation}\label{KloostermanBound}
		|Z_{m,n}^+(s)|\leq 2900\zeta^2(1+\delta)\tau(m)^{\frac12}\tau(n)^{\frac12}(1+|t|)^{\frac{1}{2}} m^{\frac34}|n|^{\frac14}. 
	\end{equation}
\end{theorem}

\begin{proof}
	We begin with a bound for $Q_{m,n}^+$. When $\re z<0$, the bound 
	\[|e^z-1|\leq 1.682|z|^{\frac14}\]
 is given by \cite[Lemma~4.1]{Oscar20}.
A
	computation as in
	\cite[Lemma~4.2]{Oscar20} shows
	\begin{align}\label{Q_bound}
		\nonumber
		|Q_{m,n}^+(s,s+2)|
		&\leq \sum_{4|c>0}\frac{\la S_{\frac12}^+(m,m,c)\ra}{c^{2\sigma}}
		\int_{-\infty}^\infty \frac{1}{(u^2+1)^{\sigma}}
		\int_0^\infty ye^{2\pi ny}\left|e\left(\frac{-m}{c^2y(u+i)}\right)-1\right|\mathrm{d}y \mathrm{d}u\\
		\nonumber
		&\leq 1.682\sum_{4|c>0}\frac{\la S_{\frac12}^+(m,m,c)\ra}{c^{2\sigma}}\int_0^\infty ye^{2\pi ny}
		\left(\frac{2\pi m}{c^2y}\right)^{\frac14}
		\int_{-\infty}^\infty\frac1{(u^2+1)^{\sigma+\frac18}}\mathrm{d}u\mathrm{d}y\\ \nonumber
		&\leq 2.663 m^{\frac14}\cdot\frac{\Gamma(\frac 74)}{(-2\pi n)^{\frac74}}\cdot\frac{\sqrt{\pi}\Gamma(\frac18)}{\Gamma(\frac58)}\sum_{4|c>0}\frac{\la S_{\frac12}^+(m,m,c)\ra}{c^{\frac32 + \delta}}\\
		&\leq 5.483 \zeta^2(1+\delta)m^{\frac14}|n|^{-\frac 74}\tau(m)^{\frac12}\tau(n)^{\frac12},
	\end{align}
where in the last step we use \eqref{TrivEstimate}. 
The theorem now follows from
	\eqref{Imn}, \eqref{Q_bound}, Proposition~\ref{prop:innerBound}, and
	\cite[Prop. 6.4]{Oscar20}:
	\[\frac{\la\Gamma(s+\frac 94)\Gamma(s-\frac 14) \ra}{|\Gamma(2s+1)|}\leq 5.84 (1+|t|)^{\frac 12}.\]
\end{proof}

\section{Perron's formula}
\label{sec:perron}
In this section we bound the partial sum of Kloosterman sums $S_{\frac12}^+$ using Perron's formula (e.g. \cite[Ch. 17]{davenport}).
See \cite[\S 7]{Oscar20}
for more details on similar computations.
\begin{theorem}\label{SumOfPlusKlSum}
	For $m>0$, $n<0$, $\delta\in (0,\frac14]$ and $\ell(\delta)$ defined as in (\ref{eq:ell}), we have
\begin{equation}\label{SumOfPlusKlSum eq}
	\la\sum_{4|c\leq x}\frac{S_{\frac12}^+(m,n,c)}{c}\ra\leq 26\,x^{\frac16+\delta}m^{\frac34} |n|^{\frac14}\tau(m)^{\frac12}\tau(n)^{\frac12}\zeta^2(1+\delta)\ell(\delta)|\log \delta|\log x. 
    \end{equation}
\end{theorem}
\begin{proof}

Note that the theorem is true for $x\leq 10000$ because $\tau(c)\leq 2\sqrt{c}$ so the left hand side is bounded by 
	\[4(m,n,4)^{\frac12} \sum_{c\leq \frac x4}(m,n,c)^{\frac12}\leq 8\sum_{d|(m,n)}d^{\frac12}\sum_{\substack{d|c\leq \frac x4}} 1 \leq 2x\tau((m,n)). \]
 Since for $x\leq 10000$, $x\leq 234x^{\frac 16}\log x$, the right hand side of \eqref{SumOfPlusKlSum eq} becomes larger by $\zeta(\frac54)^2>21$ and $\ell(\frac 14)>8$. From now on let $x> 10000$ and
	$T = x^{\frac13}$.

	Set $f(s):=Z_{m,n}^+(\frac{1+s}2)$. Then for $x\in 4\mathbb{Z}+2$ positive we have 
	\[\sum_{4|c\leq x}\frac{S_{\frac12}^+(m,n,c)}{c}=\frac1{2\pi i}\int_{\beta-i\infty}^{\beta+i\infty}f(s)\frac{x^s}{s}\mathrm{d}s, \quad \beta>\frac12. \]
	For $T>0$ we see that
	\begin{align}\label{ChangeToIntegral}
		\nonumber
		\Bigg|&\sum_{4|c\leq x}\frac{S_{\frac12}^+(m,n,c)}{c}-\frac1{2\pi i}\int_{\beta-iT}^{\beta+iT}f(s)\frac{x^s}s \mathrm{d}s\Bigg|
		=\frac1{2\pi} \Bigg|\int_{\beta-i\infty}^{\beta+i\infty}f(s)\frac{x^s}s \mathrm{d}s-\int_{\beta-iT}^{\beta+iT}f(s)\frac{x^s}s \mathrm{d}s\Bigg|\\
		&\leq \sum_{4|c>0}\frac{\la S_{\frac12}^+(m,n,c)\ra}c\(\frac xc\)^\beta\min\(1,T^{-1}\Big|\log \frac xc\Big|^{-1}\). 
	\end{align}
	Take $\beta=\frac12+\delta$ for $\delta\in(0,\frac14]$. 
	Then
	\be
	\label{eq:SHalfMinusIntegral}
	\Bigg|\sum_{4|c\leq x}\frac{S_{\frac12}^+(m,n,c)}{c}-\frac1{2\pi i}\int_{\frac12+\delta-iT}^{\frac12+\delta+iT}f(s)\frac{x^s}s \mathrm{d}s\Bigg|
	\leq \frac{x^{\frac12+\delta}}{T}\sum_{4|c>0}
	\frac{\la S_{\frac12}^+(m,n,c)\ra}{c^{\frac32+\delta}}\Big|\log \frac xc\Big|^{-1}.
	\ee

	As the proof of \cite[(16.50)]{iwkow}, for $\alpha,\beta>0$ and some non-negative function $f$ on $c$ we have
	\[\sum_{c\in[\alpha,\beta]}(m,n,c)^{\frac12}f(c)=\sum_{d|(m,n)}d^{\frac 12}\sum_{\substack{c'\in[\frac\alpha d,\frac \beta d]\\(c',\frac{(m,n)}d)=1}}f(c'd)\leq \sum_{d|(m,n)}d^{\frac 12}\sum_{c'\in[\frac\alpha d,\frac \beta d]}f(c'd), \]
	which is helpful in the following computations. 
	Using that
	$\frac{1}{\log x - \log c} \leq \frac{x}{x-c}$ for $c<x$,  
	we apply \eqref{WeilBd} to get
	\begin{align}\label{Rectangle1}
	\begin{split}
		\sum_{\substack{4|c\\\floor{\frac34 x}+1\leq c\leq x-2}}\frac{\la S_{\frac12}^+(m,n,c)\ra}{c^{\frac 32+\delta}} \la\log \frac xc\ra^{-1} 
		&\leq \frac 83\ell(\delta)\sum_{\substack{4|c\\\floor{\frac34 x}+1\leq c\leq x-2}}\frac{(m,n,c)^{\frac12}}{x-c}\\
		&\leq \frac {16}3\ell(\delta) \sum_{d|(m,n)}d^{\frac12}\sum_{t\in [\frac{3x}{16d},\frac{x-2}{4d}]}\frac1{x-4dt}\\
		&\leq \frac 43\ell(\delta)\tau((m,n))\log x. 
	\end{split}
	\end{align}
	and similarly
	\begin{align}
		\label{Rectangle2}
		\sum_{\substack{4|c\\x+2\leq c\leq\ceil{\frac54 x}-1}}\frac{\la S_{\frac12}^+(m,n,c) \ra}{c^{\frac 32+\delta}}\la\log \frac xc\ra^{-1} 
		&< \frac12\ell(\delta)\tau((m,n))\log x.
	\end{align}
	Applying \eqref{TrivEstimate} we also have
	\begin{align}
		\nonumber
		\label{Rectangle3}
		\sum_{\substack{4|c\\c\leq \frac34x}}&
		\frac{\la S_{\frac12}^+(m,n,c)\ra }{c^{\frac 32+\delta}}\la\log \frac xc\ra^{-1}
		+\sum_{\substack{4|c\\c\geq \frac54x}}
		\frac{\la S_{\frac12}^+(m,n,c)\ra}{c^{\frac 32+\delta}}\la\log \frac xc\ra^{-1}\\
		\nonumber
		&\leq \la\log \frac43 \ra^{-1} \sum_{\substack{4|c\\c\leq \frac34x}}
		\frac{\la S_{\frac12}^+(m,n,c)\ra}{c^{\frac32+\delta}}
		+\la\log \frac45\ra^{-1} \sum_{\substack{4|c\\c\geq \frac54x}}
		\frac{\la S_{\frac12}^+(m,n,c)\ra}{c^{\frac32+\delta}}\\
		&\leq 27\tau(m)^{\frac12}\tau(n)^{\frac12}\zeta^2(1+\delta). 
	\end{align}
	
	Combining 
	\eqref{eq:SHalfMinusIntegral},
	\eqref{Rectangle1}, \eqref{Rectangle2} and \eqref{Rectangle3} we get
	\begin{align}\label{Perron1}
		\begin{split}
			\la\sum_{4|c\leq x}\frac{S_{\frac12}^+(m,n,c)}{c}\ra&\leq  \frac{1}{2\pi} \la\int_{\frac12+\delta-iT}^{\frac12+\delta+iT}f(s)\frac{x^s}{s}\mathrm{d}s\ra\\
			&+(27\zeta^2(1+\delta)+2\ell(\delta)\log x) \tau(m)^{\frac12}\tau(n)^{\frac12} x^{\frac16+\delta}. 
		\end{split}
	\end{align}

 We have
 \begin{align*}
     f(s)&=\sum_{4|c}\frac{S_{\frac12}^+(m,n,c)}{c^{1+s}}=2e\(-\frac 18\)\sum_{4|c}\frac{S_{\infty,\infty}(m,n,c)}{c^{1+s}}-e\(-\frac 18\)\sum_{8|c}\frac{S_{\infty,\infty}(m,n,c)}{c^{1+s}}\\
     &:=Z_1(s)-Z_2(s). 
 \end{align*}
 By \cite{gs}, for $\re s>0$, $f$ is holomorphic because both $Z_1$ and $Z_2$ are holomorphic (the only possible pole corresponding to $\lambda_0(\frac12)=\frac 3{16}$ is excluded since $n<0$). Thus, to estimate the remaining integral in \eqref{Perron1}, we can move the path to the other three sides of the rectangle $[\delta,\frac12+\delta]\times [-T,T]$. By 
	Theorem
	\ref{BoundZmn},
	\begin{equation}
		|f(\delta+it)|\leq 2900 \zeta^2(1+\delta)m^{\frac34} |n|^{\frac14}\tau(m)^{\frac12}\tau(n)^{\frac12}(1+|\tfrac t2|)^{\frac12}
	\end{equation}
	By \eqref{TrivEstimate} we see that
	\begin{align}
		\begin{split}
			\la f\(\frac12+\delta+it\)\ra =&\la\sum_{4|c>0}\frac{S_{\frac12}^+(m,n,c)}{c^{\frac32+\delta+it}}\ra\leq 6\zeta^2(1+\delta)\tau(m)^{\frac12}\tau(n)^{\frac12}. 
	\end{split}
	\end{align}
	By the Phragm\'en-Lindel\"of principle (see
	\cite[Propositions 7.1-7.2]{Oscar20}), for $\sigma\in[\delta,\frac12+\delta]$ we have
	\begin{equation}
		\la f\(\sigma+it\)\ra \leq 2900 \zeta^2(1+\delta)m^{\frac34} |n|^{\frac14}\tau(m)^{\frac12}\tau(n)^{\frac12}(1+|\tfrac t2|)^{-\sigma+\frac12+\delta}. 
	\end{equation}
	Therefore, taking $T = x^{\frac13}$ we have
	\begin{align}\label{Integral1}
		\begin{split}
		\la \int_{\frac12+\delta+iT}^{\delta+iT}\right. &\left.f(s)\frac{x^s}{s}\mathrm{d}s\ra\\
		\leq &\,2900 \zeta^2(1+\delta)m^{\frac34} |n|^{\frac14}\tau(m)^{\frac12}\tau(n)^{\frac12}(1+\tfrac T2)^{\frac12+\delta}\int_{\delta}^{\frac12+\delta} \frac{(1+\tfrac T2)^{-\sigma }x^{\sigma}}{|\sigma+iT|}\mathrm{d}\sigma
		\\
		\leq &\,2900\zeta^2(1+\delta)m^{\frac34}|n|^{\frac14}\tau(m)^{\frac12}\tau(n)^{\frac12}(1+\tfrac T2)^{\frac12+\delta} \frac{\(\frac x{1+T/2}\)^{\frac12+\delta}}{T\log \(\frac x{1+T/2}\)}\\
		\leq &\,431\zeta^2(1+\delta) m^{\frac34} |n|^{\frac14}\tau(m)^{\frac12}\tau(n)^{\frac12}x^{\frac16+\delta}, 
	\end{split}
	\end{align}
where in the last step we used $x>10000$ and $\log (\frac x{1+T/2})>6.74$. 
The  same bound holds for the path $\delta-iT\rightarrow\frac12+\delta-iT$. 
	In the following estimates we always use $|\log \delta|\geq \log 4$ and $x>10000$ to lower the constants. Estimating as \cite[(7.10-13)]{Oscar20} we have
	\begin{equation}\label{Integral2}
		\la\int_{\delta+iT}^{\delta-iT}f(s)\frac{x^s}{s}\mathrm{d}s\ra\leq 9291 \zeta^2(1+\delta)m^{\frac34} |n|^{\frac14}\tau(m)^{\frac12}\tau(n)^{\frac12}|\log \delta|x^{\frac16+\delta}. 
	\end{equation}
	Thus,
	using \eqref{Integral1} and \eqref{Integral2} we obtain
	\bee
	\la\int_{\frac12 + \delta-iT}^{\frac12 + \delta+iT}
	f(s)\frac{x^s}{s}\mathrm{d}s\ra
	\leq
	9913\zeta^2(1+\delta)m^{\frac34} |n|^{\frac14}\tau(m)^{\frac12}\tau(n)^{\frac12}|\log \delta|x^{\frac16+\delta}. 
	\eee
	From 
	\eqref{Perron1} we have
	\begin{equation*}
	\la\sum_{4|c\leq x}\frac{S_{\frac12}^+(m,n,c)}{c}\ra\leq\(1605\zeta^2(1+\delta)m^{\frac34} |n|^{\frac14}|\log \delta|+2\ell(\delta)\log x\)\tau(m)^{\frac12}\tau(n)^{\frac12}x^{\frac16+\delta}. 
    \end{equation*}
    The result follows from the estimates $\zeta^2(1+\delta)|\log \delta|\geq 29.2$, $\ell(\delta)\geq \ell(\frac 14)> 7$ for $\delta\in(0,\frac14]$ and $x>10000$. 
\end{proof}

\section{Proof of Theorems \ref{Thm2} and \ref{Thm1}}
\label{sec:bounds}
Recall that we factor the negative discriminant $D$ into discriminants $D=dd'$ where $d$ is fundamental. From \eqref{eq:weylKloosterman} 
we have
\bee
\label{eq:Ts12}
\la\sum_{4|c\leq x} \frac{T_m(d,d';c)}{\sqrt c}\ra
\leq 
\sqrt{2}\sum_{n\mid m} n^{-\frac12}
\la
\sum_{4\mid c\leq \frac xn }\frac{S_{\frac12}^{+}(d', \frac{m^2}{n^2} d, c)}{c}
\ra.
\eee
We consider two cases: 
if $d<0$, then we can directly use 
Theorem
\ref{SumOfPlusKlSum} to estimate the inner sum on $c$. 
In the case $d'<0$ we can first use 
\eqref{eq:KloostermanRel} 
so that Theorem
\ref{SumOfPlusKlSum} applies.
With the help of the inequalities
\[\tau(n_1n_2)\leq \tau(n_1)\tau(n_2),\quad \tau(n^2)\leq \frac34 \tau(n)^2,\quad \text{and\ \ }\sum_{n|m}n^{-1-\alpha}\leq \zeta(1+\alpha),\]
when $d'>0$, we obtain
\be
\label{eq:weyllx}
\la\sum_{4|c\leq x} \frac{T_m(d,d';c)}{\sqrt c}\ra
\leq 210 \,x^{\frac16+\delta} |D|^{\frac34} 
m^{\frac12} \tau(D)\tau(m)
\zeta^2(1+\delta)\ell(\delta)
|\log \delta| \log x.
\ee
When $d'<0$, we have
\be
\label{eq:weyllx_dp less zero}
\la\sum_{4|c\leq x} \frac{T_m(d,d';c)}{\sqrt c}\ra
\leq 49 \,x^{\frac16+\delta} |D|^{\frac34} 
m^{\frac32} \tau(D)\tau(m)
\zeta^2(1+\delta)\ell(\delta)
|\log \delta| \log x.
\ee
Now consider 
$$
\la\sum_{4|c\geq x} T_m(d,d';c)\sinh\(\frac{4\pi m}{c}|D|^{\frac12}\)\ra.
$$
Using the identity
$\sinh(y) = \sqrt{\frac{\pi y}{2}} I_{\frac12}(y)$ and the estimate \cite[(4)]{ParisBessel1984}
\[I_{\frac12}(y)\leq I_{\frac12}(1)y^{\frac12}\leq 0.94y^{\frac12} \qquad\text{for } y\leq 1,\] 
when $x\geq 4\pi m |D|^{1/2}$ we obtain
\begin{multline}
	\label{eq:weylgx}
	\la\sum_{4|c\geq x} T_m(d,d';c)\sinh\(\frac{4\pi m}{c}|D|^{\frac12}\)\ra
	\\
\leq 	x^{-\frac13+\delta}|D|^{\frac 54} \tau(D)m^{\frac 32} \tau(m)
	\zeta^2(1+\delta)\ell(\delta)
	|\log \delta| \log x\cdot\left\{\begin{array}{ll}
	113313,&d'>0;\\
	26440m, &d'<0.
	\end{array} \right. 
\end{multline}
by partial summation: let $C_1\defeq C(m,D,\delta)$ where
\[s(x)=\sum_{4|c\leq x}\frac{T_m(d',d;c)}{\sqrt{c}}\quad  \text{and}\quad |s(x)|\leq C(m,D,\delta)x^{\frac16+\delta}\log x, \]
then since $T_m(d',d;c)\leq c$ by \eqref{eq:weylKloosterman}, 
\begin{align*}
    \la\sum_{4|c\geq x}\right. &\left. \frac{T_m(d',d;c)}{c}\ra \leq \frac{|s(x)|}{\sqrt{x}}+\la\int_x^\infty t^{-\frac12}\mathrm{d}s(t)\ra\\
    &\leq 2C_1x^{-\frac13+\delta}\log x+\frac12C_1\int_x^\infty t^{-\frac 43+\delta}\log t\mathrm{d}t\\
    &\leq \(2+\frac{1/2}{\frac13-\delta}\)C_1x^{-\frac13+\delta}\log x+\frac{1/2}{\frac13-\delta}C_1\int_x^\infty t^{-\frac 43+\delta}\mathrm{d}t\\
    &\leq \(2+6+\frac{72}{\log 4\pi}\)C_1x^{-\frac13+\delta}\log x,
\end{align*}
where in the last step we use $x\geq 4\pi$ and $\frac13-\delta\geq \frac1{12}$. Theorem~ \ref{Thm2} 
now follows from
Lemma~\ref{lem:TrWeyl} 
and \eqref{eq:weylgx}.
After determining the constant $\ell(\frac 15)= 28.117$,  
one also obtains Corollary~\ref{Coro3}.

Similarly, for $x\geq 2 m |D|^{\frac12}$ one obtains 
\[I_{\frac12}(y)\leq \frac{I_{\frac12}(2\pi)}{\sqrt{2\pi}}y^{\frac12}\qquad\text{for } y\leq 2\pi\]
by \cite[(4)]{ParisBessel1984} and
\begin{multline}
	\label{eq:weylgx2}
	\la\sum_{4|c\geq x} T_m(d,d';c)\sinh\(\frac{4\pi m}{c}|D|^{\frac12}\)\ra
	\\
\leq 	x^{-\frac13+\delta}
	|D|^{\frac 54} \tau(D) m^{\frac 32} \tau(m)
	\zeta^2(1+\delta)\ell(\delta)
	|\log \delta| \log x\cdot\left\{\begin{array}{lc}
	9\times 10^6,  &d'>0;\\
	2.1\times 10^6m, &d'<0.
	\end{array} \right. 
\end{multline}
Theorem~\ref{Thm1} 
follows from
Lemma~\ref{lem:TrWeyl},
\eqref{eq:weyllx}, \eqref{eq:weyllx_dp less zero} and
\eqref{eq:weylgx2} and the bound below: 
\begin{align*}
\begin{split}
    &\la\sum_{4|c< 2m|D|^{\frac12}}T_m(d',d;c) \exp\(-\frac{4\pi m}c |D|^{\frac12}\)\ra\leq\int_{2}^{2m|D|^{\frac12}} \sqrt{t}\,e^{-\frac{4\pi m}t |D|^{\frac12}}\mathrm{d}s(t)\\
    &\leq C_1e^{-2\pi }(2m|D|^{\frac12})^{\frac 23+\delta}\log(2m|D|^{\frac 12})+C_1 e^{-2\pi} \int_{2}^{2m|D|^{\frac12}} t^{\frac 16+\delta}\(\frac{1}{2\sqrt t}+\frac{4\pi m|D|^{\frac 12}}{t^{\frac 32}}\)\log t \mathrm{d}t\\
    &\leq 80C_1e^{-2\pi }(2m|D|^{\frac12})^{\frac 23+\delta}\log(2m|D|^{\frac 12})+C_1e^{-2\pi} \int_{2}^{2m|D|^{\frac12}}\( \frac34t^{-\frac13+\delta}+\frac{48\pi m|D|^{\frac12}}{t^{\frac 43-\delta}}\)\mathrm{d}t\\
    &\leq 986C_1e^{-2\pi} (2m|D|^{\frac12})^{\frac 23+\delta}\log(2m|D|^{\frac 12})\\
    &\leq m^{\frac 76+\delta} |D|^{\frac{13}{12}+\frac\delta 2}\tau(m)\tau(D)\zeta^2(1+\delta )\ell(\delta )|\log \delta |\log(2m|D|^{\frac 12})\cdot\left\{\begin{array}{ll}
         735& d'>0, \\
         172m& d'<0.  
    \end{array}
    \right. 
\end{split}
\end{align*}

\noindent {\bf Acknowledgements.}
The authors thank Scott Ahlgren 
for a lot of delightful discussions and suggestions and thank Nick Andersen for insightful comments on our result. 
The first author was partially supported by
the Alfred P. Sloan Foundation's MPHD Program, awarded in 2017.

\noindent{\bf Competing interests. }
The authors declare that they have no known competing financial interests or personal relationships that could have appeared to influence the work reported in this paper. 

\noindent{\bf Data Availability Statement. }
 Data sharing not applicable to this article as no datasets were generated or analysed during the current study.
	
\bibliographystyle{abbrv}
\bibliography{singular} 

\end{document}